\documentclass[a4paper]{amsart}

\usepackage{amssymb}

\theoremstyle{plain}
\newtheorem{thm}{Theorem}[section]
\newtheorem{prop}[thm]{Proposition}
\newtheorem{lem}[thm]{Lemma}

\newtheorem{conj}[thm]{Conjecture}

\theoremstyle{definition} 
\newtheorem{defn}[thm]{Definition}
\newtheorem{rem}[thm]{Remark}
\newtheorem{ex}[thm]{Example}

\newtheorem*{ack}{Acknowledgments} 

\begin{document}



\title[Characterization of $2$-dimensional SLC hypersurfaces]{Characterization of two-dimensional semi-log canonical hypersurfaces in arbitrary characteristic}

\author{Kohsuke Shibata}

\address{Graduate School of Mathematical Sciences, University of Tokyo, 3-8-1 
Komaba, Meguro-ku, 
Tokyo, 153-8914, Japan.}

\email{shibata@ms.u-tokyo.ac.jp}


\begin{abstract}
In this paper we characterize two-dimensional semi-log canonical hypersurfaces in arbitrary characteristic
from the viewpoint of the initial term of the defining equation.
As an application, we prove a
conjecture about a uniform bound of divisors computing minimal log discrepancies
for two dimensional varieties, which is a conjecture by Ishii and also a special case
of the conjecture by Musta\c{t}\v{a}-Nakamura.
\end{abstract}

\maketitle

\section{Introduction}
Log canonical singularities play important roles in the minimal model program. 
These singularities are normal, as is well known. 
In \cite{KS},
Koll\'{a}r and Shepherd-Barron  introduced a semi-log canonical singularity which is a kind of generalization of log canonical singularity to non-normal
one.
The classification of two-dimensional semi-log canonical  hypersurfaces in characteristic $0$ is given by Liu and Rollenske in \cite{LR}.

The aim of this paper is to understand two-dimensional semi-log canonical  hypersurfaces in arbitrary characteristic from the viewpoint of the initial term of the defining equation.

\begin{thm}\label{Main Theorem of slc}
Let $k$ be an algebraically closed field of
arbitrary characteristic, $X=\mathrm{Spec}k[[x,y,z]]/(f)$ and $f\in (x,y,z)$. 
Then the following are equivalent:
\begin{enumerate}
\item $X$ is semi-log canonical at the origin,
\item There exist a $k$-automorphism  $\phi$ of $k[[x,y,z]]$ and $w=(w_1,w_2,w_3)\in \mathbb N^3$ such that 
$\mathrm{Spec}k[[x,y,z]]/(\mathrm{in}_w\phi(f))$ is semi-log canonical at the origin and $w$ is  one of the following:
\begin{align*}
\mathrm{(i)}&\ (1,1,1),& \mathrm{(ii)}&\ (3,2,2),& \mathrm{(iii)}&\ (2,1,1),& \mathrm{(iv)}&\ (6,4,3),\\
 \mathrm{(v)}&\ (9,6,4),&\mathrm{(vi)}&\ (15,10,6),& \mathrm{(vii)}&\ (3,2,1).
\end{align*}
\end{enumerate}
\end{thm}

To understand semi-log canonical singularities, we need to consider minimal log discrepancies.
Minimal log discrepancies are invariants of singularities appearing in the minimal model program.

To prove Theorem \ref{Main Theorem of slc},  we investigate minimal log discrepancies for a pair of a  $3$-dimensional smooth variety and a principal ideal  from the viewpoint of the initial ideal.

\begin{thm}\label{Main Theorem of mld}
Let $k$ be an algebraically closed field of
arbitrary characteristic,  $A=\mathrm{Spec}k[[x,y,z]]$,  $0$ be the origin of $A$ and $f\in k[[x,y,z]]\setminus\{0\}$.
Then there exist a $k$-automorphism  $\phi$ of $k[[x,y,z]]$ and $w=(w_1,w_2,w_3)\in \mathbb N^3$ such that 
$$\mathrm{mld}(0;A,(f))=\mathrm{mld}\big(0;A,(\mathrm{in}_w\phi(f))\big)$$
and $w$ is  one of the following:
\begin{align*}
\mathrm{(i)}&\ (1,1,1),& \mathrm{(ii)}&\ (3,2,2),&\mathrm{(iii)}&\ (2,1,1),& \mathrm{(iv)}&\ (6,4,3),& \mathrm{(v)}&\ (9,6,4),\\
\mathrm{(vi)}&\ (15,10,6),& \mathrm{(vii)}&\ (3,2,1),&\mathrm{(viii)}&\ (10,5,4),& \mathrm{(ix)}&\ (15,8,6),& \mathrm{(x)}&\ (21,14,6).
\end{align*}
Moreover if $\mathrm{ord}_{(1,1,1)}f\neq 3$ or $\mathrm{mld}(0;A,(\mathrm{in}_{(1,1,1)}f))\ge 0$, then
$E_w$ computes  $\mathrm{mld}\big(0;A,(\phi(f))\big)$ and $\mathrm{mld}\big(0;A,(\mathrm{in}_w\phi(f))\big)$,
where $E_w$ is the toric divisor over $A$ corresponding to $w$.
In  case $\mathrm{(ii),(iii)},\dots,\mathrm{(vii)}$, $\mathrm{mld}(0;A,(f))\ge 0.$
In  case $\mathrm{(viii)}$, $\mathrm{(ix)}$,  $\mathrm{(x)}$, $\mathrm{mld}(0;A,(f))=-\infty.$
\end{thm}

In \cite{I}, Ishii posed the following conjecture, which is a special case
of  Musta\c{t}\v{a}-Nakamura's conjecture (\cite{MN}).

\begin{conj}[Conjecture $D_n$] 
 For  $n\in \mathbb N$, there exists $M_n\in \mathbb N$ depending only on $n$
such that 
for any $n$-dimensional variety $X$ and a closed point $x\in X$ with a closed immersion $X\subset A$ around $x$ into a smooth variety $A$ of dimension $\mathrm{emb}(X,x)\le 2n$,
 there exists a prime divisor $E$ over $A$ with the
center at $x$ 
and $k_E \leq M_n$ such that
$$\left\{\begin{array}{l}
a(E;  A, I_X^c)= \mathrm{mld} (0; A, I_X^c)\geq 0,  \ {\mbox or}\\
a (E; A, I_X^c)< 0\ \ {\mbox if }\ \   \mathrm{mld} (0; A, I_X^c)= - \infty. \\
 \end{array}
 \right.$$
 Here, $\mathrm{emb}(X,x)$ is the embedding dimension of $X$ at $x$, $I_X$ is the defining ideal of $X$ and $c=\mathrm{emb}(X,x)-n$.

\end{conj}

It is known that this conjecture  bounds the  number of blow-ups to
obtain a prime divisor computing the minimal log discrepancy.
In \cite{I},  Conjecture $D_1$ is proved  for arbitrary characteristic and one can take   $M_1=4$.
On the other hand,  Conjecture $D_2$ is also proved in \cite{I} for characteristic $p\neq 2$ and one can take  $M_2\leq 58$ in this case.

In this paper, we prove the same result for arbitrary characteristic, as an application of Theorem \ref{Main Theorem of mld}.

 \begin{thm}
  Conjecture $D_2$ holds and we can take  $M_2\leq 58$ in arbitrary characteristic.
\end{thm}

The structure of this paper is as follows: 
In Section 2 we give the definitions  and some basic properties which will be used in this paper.
In Section 3 we characterize two-dimensional semi-log canonical hypersurfaces in arbitrary characteristic
from the viewpoint of the initial term of the defining equation.
In Section 4 we give a proof of Conjecture $D_2$ in arbitrary characteristic.

\begin{ack}
The author would like to thank Professor  Shihoko Ishii  and Professor  Shunsuke Takagi for valuable conversations.
The author is partially supported by JSPS Grant-in-Aid for Early-Career Scientists 19K14496 and the Iwanami Fujukai Foundation.
\end{ack}

\noindent 
{\bf Conventions.}
Throughout this paper,  $k$ is an algebraically closed field  of arbitrary characteristic and a variety is a reduced pure-dimensional scheme of finite type over $k$ or a reduced pure-dimensional scheme $\mathrm{Spec}k[[x_1,\dots,x_n]]/I$ for some ideal $I$.

\section{Preliminaries}

In this section, we give necessary definitions and record various properties for later use.

\subsection{Minimal log discrepancies and semi-log canonical singularities}
\begin{defn}
Let $X$ be a variety.
We say that $E$ is a prime divisor over $X$, if there is a birational morphism $Y\to X$ such that $Y$ is normal and $E$ is a prime divisor on $Y$.
The closure of $f(E)\subset X$ is called the center of $E$ on $X$  and denoted by $c_X(E)$.
\end{defn}

\begin{defn}
Let $X$ be a $\mathbb Q$-Gorenstein variety which satisfies Serre's condition $S_2$ and is Gorenstein in codimension $1$, $\mathfrak a\subset \mathcal O_X$ be a non-zero ideal sheaf  and $E$ be a prime divisor over $X$.
The log discrepancy of $(X,\mathfrak a)$ at $E$ is defined as 
$$a(E;X,\mathfrak a):=k_E-\mathrm{ord}_E(\mathfrak a)+1,$$
where $k_E$ is the coefficient of the relative canonical divisor $K_{Y/X}$ at $E$.
Here $f: Y\to X$ is birational morphism with normal $Y$ such that $E$ appears on $Y$.
\end{defn}

\begin{defn}
Let $X$ be a $\mathbb Q$-Gorenstein variety which satisfies Serre's condition $S_2$ and is Gorenstein in codimension $1$ and $\mathfrak a\subset \mathcal O_X$ be a non-zero ideal sheaf.
The minimal log discrepancy of the pair $(X,\mathfrak a)$ at a  closed subset $W\subset X$ is defined as follows:
$$\mathrm{mld}(W;X,\mathfrak a)=\mathrm{inf}\{a(E;X,\mathfrak a)\ |\ E: \mbox{prime\ divisors\ over\ } X\ \mbox{with\ the\ center\ in\ } W\}$$
when $\dim X\geq2$.
When $\dim X=1$ and the right-hand side is $\ge 0$, then we define $\mathrm{mld}(W;X,\mathfrak a)$ by the
right-hand side.
Otherwise, we define $\mathrm{mld}(W;X,\mathfrak a)=-\infty$.
We shall simply write $\mathrm{mld}(W;X)$ instead of $\mathrm{mld}(W;X,\mathcal O_X)$.
\end{defn}

\begin{defn}
Let  $X$ be a normal $\mathbb Q$-Gorenstein variety, $\mathfrak a\subset \mathcal O_X$ be a non-zero ideal sheaf and  $x$ be a closed point of $X.$
$(X,\mathfrak a)$ is said to be log canonical at $x$ if 
 $a(E;X,\mathfrak a)\ge 0$ for every  prime divisor $E$ over $X$  whose center contains $x$.
\end{defn}

\begin{defn}
Let  $X$ be a $\mathbb Q$-Gorenstein variety which satisfies Serre's condition $S_2$ and is Gorenstein in codimension $1$
and $x$ be a closed point of $X.$
$X$ is said to be semi-log canonical at $x$ if 
 $a(E;X,\mathcal O_X)\ge 0$ for every  prime divisor $E$ over $X$  whose center contains $x$.
\end{defn}

\begin{defn}
A normal  surface singularity is said to be a  simple elliptic singularity  if the exceptional divisor of the minimal resolution is a smooth elliptic curve.
\end{defn}

\begin{rem}
By the definition, simple elliptic singularities are log canonical singularities.
\end{rem}

\begin{defn}
Let $X$ be a $\mathbb Q$-Gorenstein variety which satisfies Serre's condition $S_2$ and is Gorenstein in codimension $1$, $W$ be a closed subset of $X$ and $\mathfrak a\subset \mathcal O_X$ be a non-zero ideal sheaf.
We say that a prime divisor $E$ over $X$ with the center in $W$ computes $\mathrm{mld}(W;X,\mathfrak a)$ if either 
$$a(E;X,\mathfrak a)=\mathrm{mld}(W;X,\mathfrak a)\ge0\ \  \mathrm{or}\ \ a(E;X,\mathfrak a)<0.$$
\end{defn}

\subsection{F-pure rings vs. log canonical singularities}

Let $R$ be a domain of characteristic $p>0$ and $F: R\to R$ the Frobenius map which sends $x\in R$ to $x^p\in R$.
For an integer $e\ge0$, we can identify $e$-times iterated Frobenius map $F^e: R\to R^{1/p^e}$ with the natural inclusion map $R\hookrightarrow R^{1/p^e}$.
The ring $R$ is called $F$-finite, if $F:R\to R$ is a finite map.
For example, a complete local ring with perfect residue field is $F$-finite.  

\begin{defn}
Let $(R,\mathfrak m)$ be  an $F$-finite local domain and $\mathfrak a$ be an ideal of $R$. 
We say the pair $(R,\mathfrak a)$ is $F$-pure if there exist an integer $e>0$ and an element $a\in \mathfrak a^{p^e-1}$ such that the inclusion $a^{1/p^e}R\hookrightarrow R^{1/p^e}$ splits as an $R$-module homomorphism. 

\end{defn}

\begin{prop}[Fedder criterion {\cite[Proposition 2.1]{F}}, see {\cite[Corollary 2.7]{HW}}]\label{Fedder}
Let $K$ be a perfect field of characteristic $p>0$,  $R=K[[x_1,\dots,x_n]]$ and $f\in (x_1,\dots,x_n)$. 
Then $(R,(f))$ is $F$-pure if and only if $f^{p-1}\notin (x_1^{p},\dots, x_n^{p})$.
\end{prop}

\begin{defn}
Let $K$ be a perfect field of characteristic $p>0$, $A=\mathrm{Spec} K[[x_1,\dots,x_n]]$,  $f\in (x_1,\dots,x_n)$ and $0$ be the origin of $A$.
Then we say $(A,(f))$ is $F$-pure at the origin if the pair $(\mathcal O_{A,0},(f))$ is $F$-pure.
\end{defn}

\begin{thm}[{\cite[Theorem 3.3]{HW}}]\label{HW thm}
Let $K$ be a perfect field of characteristic $p>0$, $A=\mathrm{Spec} K[[x_1,\dots,x_n]]$ and  $f\in (x_1,\dots,x_n)$.
If $(A,(f))$ is $F$-pure at the origin, then $(A,(f))$ is log canonical at the origin.
\end{thm}

\subsection{Initial term with respect to $w$}

\begin{defn}
Let $w\in \mathbb N^n$,
 $f=\sum a_mx^m\in k[[x_1,\dots,x_n]]\setminus\{0\}$ and $\mathfrak a$ be an ideal of $k[[x_1,\dots,x_n]]$.
Then we  define 
$$\mathrm{supp} f=\{m\in \mathbb Z^n_{\ge 0}\ |\ a_m\neq 0\},$$
$$\mathrm{ord}_w f =\mathrm{min}\{w\cdot m\ |\ m\in\mathrm{supp}f\},$$ where the dot product $\cdot$ denotes the standard inner product on $\mathbb R^n$.
Then we define the initial term of $f$ with respect to $w$ to be 
$$\mathrm{in}_wf:=\sum_{m\in\mathrm{supp}f,\ w\cdot m=\mathrm{ord}_wf}a_mx^m.$$
We define $\mathrm{ord}_w 0=\infty$ and $\mathrm{in}_w 0=0$. 
The initial ideal of $\mathfrak a$ with respect to $w$ is defined by
$$\mathrm{in}_w\mathfrak a=\{\mathrm{in}_wf\ |\ f\in \mathfrak a\}.$$
\end{defn}


\subsection{Arc spaces and minimal log discrepancies}
We briefly review in this section the  results of arc spaces and minimal log discrepancies in \cite{IR2}.
For simplicity, we consider only the case when a scheme is $\mathrm{Spec} k[[x_1,\dots,x_n]]$.
We remark that  in \cite{IR2}, we assume that  a scheme is of finite type over a field. 
However, the proofs in \cite{IR2} also work for $\mathrm{Spec} k[[x_1,\dots,x_n]]$.

\begin{defn}
Let $f\in (x_1,\dots,x_n)$ be an element of $k[[x_1,\dots,x_n]]$.
Then 
$$f^{(m)}\in k[[x_i^{(j)}\ |\ i=1,\dots,n,j=0,\dots,m]],$$
is defined as follows:
$$f\Big(\sum_{j=0}^\infty x_1^{(j)}t^j,\dots,\sum_{j=0}^\infty x_n^{(j)}t^j\Big)=\sum_{j=0}^\infty f^{(j)}t^j.$$
\end{defn}

\begin{defn}
Let $\mathfrak a=(f_1,\dots,f_r)\subset (x_1,\dots,x_n)$ be an ideal of $k[[x_1,\dots,x_n]]$,  $A=\mathrm{Spec} k[[x_1,\dots,x_n]]$ and $0$ be the origin of $A$.
Then the scheme $$A_\infty=\mathrm{Spec} k[[x_i^{(j)}\ |\ i=1,\dots,n,j\in\mathbb Z_{\ge 0} ]]$$ is called the arc space of $A$.
Then we define for $m\in \mathbb Z_{\ge 0}$,
$$\mathrm{Cont}^{\ge m+1}(\mathfrak a)=\mathrm{Spec} k[[x_i^{(j)}\ |\ i=1,\dots,n,j\in\mathbb Z_{\ge 0}]]/(f_s^{(l)}\ |\ 1\le s\le r,\ 0\le l\le m).$$
\end{defn}

\begin{thm}[{\cite[Theorem 7.4]{EM}}, {\cite[Theorem 3.18]{IR2}}]\label{mld contact in IR}
Let $\mathfrak a=(f_1,\dots,f_r)\subset (x_1,\dots,x_n)$ be a non-zero ideal of $k[[x_1,\dots,x_n]]$,  $A=\mathrm{Spec} k[[x_1,\dots,x_n]]$ and $0$ be the origin of $A$.
Then
\[\mathrm{mld}(0;A,\mathfrak a)=\inf_{m\in \mathbb N}\Big\{\mathrm{codim}\Big(\mathrm{Cont}^{\ge m}(\mathfrak a)\cap\mathrm{Cont}^{\ge 1}((x_1,\dots,x_n)),A_\infty\Big)-m\Big\}.\]

\end{thm}

\begin{thm}[inversion of adjunction, {\cite[Theorem 8.1]{EM}}, {\cite[Theorem 3.23]{IR2}}]\label{inversion of adjunction}
Let $f$ be a non-zero element of $(x_1,\dots,x_n)$ and $0$ be the origin of $\mathrm{Spec} k[[x_1,\dots,x_n]]$.
Then $$\mathrm{mld}(0; \mathrm{Spec} k[[x_1,\dots,x_n]],(f))=\mathrm{mld}(0;\mathrm{Spec} k[[x_1,\dots,x_n]]/(f)).$$

\end{thm}

\section{Characterization of $2$-dimensional Semi-log canonical hypersurfaces}

In this section,  we characterize two-dimensional semi-log canonical hypersurfaces in arbitrary characteristic
from the viewpoint of the initial term of the defining equation.

\begin{prop}\label{slc mld}
Let  $A=\mathrm{Spec}k[[x,y,z]]$, $f\in (x,y,z)\setminus\{0\}$ and $X=\mathrm{Spec}k[[x,y,z]]/(f)$.
Then  $X$ is semi-log canonical at the origin if and only if
$\mathrm{mld}(0; A,(f))\ge 0,$
where $0$ is the origin of $A$.
\end{prop}

\begin{proof}
Note that there exists a log resolution of $(X,\mathfrak m)$, where $\mathfrak m$ is the defining ideal of $0$.
Hence in the same way as characteristic 0 case,
$X$ is semi-log canonical at $0$ if and only if $\mathrm{mld}(0;X)\ge 0$ (for example, see \cite[Page 532]{EM}).
By Theorem \ref{inversion of adjunction}, we have $\mathrm{mld}(0; A,(f))=\mathrm{mld}(0;X)$.
Therefore this proposition holds.

\end{proof}

\begin{lem}\label{mld contact}
Let $\mathfrak a=(f_1,\dots,f_r)\subset (x_1,\dots,x_n)$ be a non-zero ideal of $k[[x_1,\dots,x_n]]$,  $A=\mathrm{Spec} k[[x_1,\dots,x_n]]$ and $0$ be the origin of $A$.
Then
\[\mathrm{mld}(0;A,\mathfrak a)=\inf_{m\in \mathbb N}\Big(\mathrm{ht}(x_1^{(0)},\dots,x_n^{(0)},f_1^{(0)},\dots,f_r^{(0)},\dots,f_1^{(m-1)},\dots,f_r^{(m-1)})-m\Big).\]
\end{lem}

\begin{proof}
Note that 
\begin{align*}
&\mathrm{Cont}^{\ge m}(\mathfrak a)\cap\mathrm{Cont}^{\ge 1}((x_1,\dots,x_n))\\
=&\mathrm{Spec}k[[x_i^{(j)}\ |\ i=1,\dots,n,j\in\mathbb Z_{\ge 0}]]/(x_1^{(0)},\dots,x_n^{(0)},f_s^{(l)}\ |\ 1\le s\le r,\ 0\le l\le m-1).
\end{align*}
Therefore 
by Theorem \ref{mld contact in IR}
we have 
\[\mathrm{mld}(0;A,\mathfrak a)=\inf_{m\in \mathbb N}\Big(\mathrm{ht}(x_1^{(0)},\dots,x_n^{(0)},f_1^{(0)},\dots,f_r^{(0)},\dots,f_1^{(m-1)},\dots,f_r^{(m-1)})-m\Big).\]
\end{proof}

The following proposition is the main tool in this section.
This proposition is useful to decide whether a  pair of a smooth variety and an ideal is log canonical.

\begin{prop}\label{mld initial}
Let $\mathfrak a$ be a non-zero ideal of $k[[x_1,\dots,x_n]]$ and $A=\mathrm{Spec}k[[x_1,\dots,x_n]]$.
Then for any $w\in \mathbb N^n$,
$$\mathrm{mld}(0; A,\mathfrak a)\ge \mathrm{mld}(0; A, \mathrm{in}_{w}\mathfrak a),$$
where $0$ is the origin of $A$.
\end{prop}

\begin{proof}
If $\mathfrak a=k[[x_1,\dots,x_n]]$, then $\mathrm{in}_{w}\mathfrak a=k[[x_1,\dots,x_n]]$.
Therefore this proposition holds.

We assume that $\mathfrak a\subset (x_1,\dots,x_n)$.
There exist $f_1,\dots,f_r\in k[[x_1,\dots,x_n]]$  such that $\mathfrak a=(f_1,\dots,f_r)$ and $\mathrm{in}_w\mathfrak a=(\mathrm{in}_wf_1,\dots,\mathrm{in}_wf_r)$.
Note that for $s,m\in \mathbb N$ with $1\le s\le r$, we can regard $f_s^{(0)},\dots,f_s^{(m-1)}$ as elements of $k[[x_i^{(l)}\ |\ i=1,\dots,n,l=0,\dots,m-1]]$.
Let $w^{(m)}=(w,\dots,w)\in\mathbb N^{mn}$.
We can consider the order of an element of $k[[x_i^{(l)}\ |\ i=1,\dots,n,l=0,\dots,m-1]]$ with respect to $w^{(m)}$ such that $\mathrm{ord}_w(x_i)=\mathrm{ord}_{w^{(m)}}(x_i^{(l)})$. 
If $(\mathrm{in}_wf_i)^{(j)}\neq 0$, then we have $(\mathrm{in}_wf_i)^{(j)}=\mathrm{in}_{w^{(m)}}f_i^{(j)}$ since $\mathrm{ord}_{w^{(m)}}f_i^{(j)}=\mathrm{ord}_{w}f_i$.
Therefore by \cite[Theorem 15.17]{E}
\begin{align*}
&\mathrm{ht}(x_1^{(0)},\dots,x_n^{(0)},f_1^{(0)},\dots,f_r^{(0)},\dots,f_1^{(m-1)},\dots,f_r^{(m-1)})\\
\ge&\mathrm{ht}(x_1^{(0)},\dots,x_n^{(0)},\mathrm{in}_{w^{(m)}}f_1^{(0)},\dots,\mathrm{in}_{w^{(m)}}f_r^{(0)},\dots,\mathrm{in}_{w^{(m)}}f_1^{(m-1)},\dots,\mathrm{in}_{w^{(m)}}f_r^{(m-1)})\\
\ge&\mathrm{ht}(x_1^{(0)},\dots,x_n^{(0)},(\mathrm{in}_{w}f_1)^{(0)},\dots,(\mathrm{in}_{w}f_r)^{(0)},\dots,(\mathrm{in}_{w}f_1)^{(m-1)},\dots,(\mathrm{in}_{w}f_r)^{(m-1)}).
\end{align*}
Hence by Theorem \ref{mld contact}, we have
$$\mathrm{mld}(0; A,\mathfrak a)\ge \mathrm{mld}(0; A, \mathrm{in}_{w}\mathfrak a).$$
\end{proof}

Let us denote by $E_w$ the toric  prime divisor over $\mathrm{Spec}k[[x,y,z]]$ corresponding to the 1-dimensional cone $\mathbb R_{\geq 0} w$ for $w\in \mathbb Z^3_{\geq 0}\setminus \{(0,0,0)\}$.
Then, $w\cdot (1,1,1)=k_{E_w}+1$ and $\mathrm{ord}_{E_w}(f)=\mathrm{min}\{w\cdot (a,b,c) \ |\ x^ay^bz^c\in\mathrm{Supp}f\}$ for $f\in k[[x,y,z]]$.

\begin{lem}\label{slc order 3}
Let $A=\mathrm{Spec}k[[x,y,z]]$, $0$ be the origin of $A$ and $f\in k[[x,y,z]]\setminus\{0\}$.
Suppose that $\mathrm{ord}_{(1,1,1)}f\ge 3$, then 
$$\mathrm{mld}(0;A,(f))=\mathrm{mld}(0;A,(\mathrm{in}_{(1,1,1)}f)).$$
Moreover if $\mathrm{ord}_{(1,1,1)}f\neq 3$ or $\mathrm{mld}(0;A,(\mathrm{in}_{(1,1,1)}f))\ge 0$,
then $E_{(1,1,1)}$ computes $\mathrm{mld}(0;A,(f))$ and $\mathrm{mld}(0;A,(\mathrm{in}_{(1,1,1)}f)).$
\end{lem}

\begin{proof}
Let $w=(1,1,1)$.
We assume that $\mathrm{ord}_wf\ge 4$.
Then
$$a(E_{w};A,(f))=a(E_{w};A,(\mathrm{in}_wf))<0.$$
Therefore $E_{(1,1,1)}$ computes $\mathrm{mld}(0;A,(f))$ and $\mathrm{mld}(0;A,(\mathrm{in}_{(1,1,1)}f))$ and we have
$$\mathrm{mld}(0;A,(f))=\mathrm{mld}(0;A,(\mathrm{in}_wf))=-\infty.$$

We assume that $\mathrm{ord}_wf=3$.
Then by \cite[Corollary 5.16]{I}, we have 
$$\mathrm{mld}(0;A,(f))=\mathrm{mld}(0;A,(\mathrm{in}_wf)).$$
Note that $a(E_{w};A,(f))=a(E_{w};A,(\mathrm{in}_wf))=0$.
If $\mathrm{mld}(0;A,(\mathrm{in}_{(1,1,1)}f))\ge 0$, then $E$ computes $\mathrm{mld}(0;A,(f))$ and $\mathrm{mld}(0;A,(\mathrm{in}_{(1,1,1)}f))$ by Proposition \ref{mld initial}.
\end{proof}

The following lemma will be used in  the proof of Proposition \ref{key prop}.

\begin{lem}\label{ord 4 chk = 2}
Let $A=\mathrm{Spec}k[[x,y,z]]$, $0$ be the origin of $A$ and $f\in k[[x,y,z]]$.
Suppose that $\mathrm{in}_{(1,1,1)}f=\mathrm{in}_{(3,2,2)}f=x^2$
 and $\mathrm{ch}k= 2$. 
Then there exist a $k$-automorphism  $\phi$ of  $k[[x,y,z]]$ and $w=(w_1,w_2,w_3)\in \mathbb N^3$ such that 
$E_w$ computes  $\mathrm{mld}\big(0;A,(\phi(f))\big)$ and $\mathrm{mld}\big(0;A,(\mathrm{in}_w\phi(f))\big)$,
$$\mathrm{mld}(0;A,(f))=\mathrm{mld}\big(0;A,(\mathrm{in}_w\phi(f))\big)$$
and $w$ is  one of the following:
\begin{align*}
\mathrm{(i)}&\ (10,5,4),& \mathrm{(ii)}&\ (2,1,1),& \mathrm{(iii)}&\ (15,8,6).
\end{align*}
In  case $\mathrm{(ii)}$, $\mathrm{mld}(0;A,(f))=0.$
In  case $\mathrm{(i)\ and\  (iii)}$, $\mathrm{mld}(0;A,(f))=-\infty.$
\end{lem}

\begin{proof}

Let $h$ be the element of $k[[x,y,z]]$ such that $f=x^2+h$.
Note that $\mathrm{ord}_{(2,1,1)}f\ge 4$ since
$\mathrm{in}_{(1,1,1)}f=\mathrm{in}_{(3,2,2)}f=x^2$.
Therefore $\mathrm{ord}_{(2,1,1)}h\ge 4$.

\noindent
{\bf Step 1.} 
Assume $\mathrm{ord}_{(2,1,1)}h\ge 5$.\\
Note that $\mathrm{ord}_{(10,5,4)}f=20$ since $\mathrm{ord}_{(2,1,1)}h\ge5$.
Then  for  the prime divisor $E_{(10,5,4)}$,
 we have $a(E_{(10,5,4)};A,(f))=a(E_{(10,5,4)};A,(\mathrm{in}_{(10,5,4)}f))=-1$.
Therefore $E_{(10,5,4)}$ computes $\mathrm{mld}(0; A, (f))$ and $\mathrm{mld}(0; A, (\mathrm{in}_{(10,5,4)}f))$ and we have
  $$\mathrm{mld}(0; A, (f))=\mathrm{mld}(0; A, (\mathrm{in}_{(10,5,4)}f))=-\infty.$$

Note that if  $\mathrm{ord}_{(2,1,1)}h=4$,  we have
$$\mathrm{in}_{(2,1,1)}h=\alpha_1xyz+\alpha_2xy^2+\alpha_3xz^2+\alpha_4y^4+\alpha_5y^3z+\alpha_6y^2z^2+\alpha_7yz^3+\alpha_8z^4$$
for some $\alpha_1,\dots,\alpha_8\in k$.
By a coordinate transformation 
$$y\mapsto \beta_1y+\beta_2z,\ z\mapsto \beta_3y+\beta_4z\ \mathrm{for\ some}\ \beta_1,\beta_2,\beta_3,\beta_4\in k,$$
we may assume that  $\alpha_3=0$.
By a coordinate transformation 
$$x\mapsto x+\gamma y^2\ \mathrm{for\ some}\ \gamma\in k,$$
we may assume that  $\alpha_4=0$.

We assume that 
$$\mathrm{in}_{(2,1,1)}h=a_1xyz+a_2xy^2+a_3y^3z+a_4y^2z^2+a_5yz^3+a_6z^4$$
for some $a_1,\dots,a_6\in k$.

\vskip.3truecm
\noindent
{\bf Step 2.} Assume   $a_1\neq0$.\\
The pairs $(A, (f))$ and $(A, (\mathrm{in}_{(2,1,1)}f))$ are F-pure at the origin by Proposition \ref{Fedder}.
Therefore $(A, (f))$ and $(A, (\mathrm{in}_{(2,1,1)}f))$ are log canonical at the origin by Theorem \ref{HW thm}.
Thus $E_{(2,1,1)}$ computes $\mathrm{mld}(0; A, (f))$ and $\mathrm{mld}(0; A, (\mathrm{in}_{(2,1,1)}f))$ and
we have
$$\mathrm{mld}(0; A, (f))=\mathrm{mld}(0; A, (\mathrm{in}_{(2,1,1)}f))=0.$$

\vskip.3truecm
\noindent
{\bf Step 3.} Assume  $a_1=0$ and $a_2\neq0$.\\
Then by a coordinate transformation
$$x\mapsto x+b_1yz,\ y\mapsto b_2y,\ z\mapsto b_3z\ \ \mathrm{for\ some}\ b_1,b_2,b_3\in k,$$
we may assume that 
$$\mathrm{in}_{(2,1,1)}f=x^2+xy^2+y^3z+c_1y^2z^2+c_2yz^3+c_3z^4$$
for some  $c_1,c_2,c_3\in k$.
Then $\mathrm{in}_{(2,1,1)}f$ is the defining equation of  a simple elliptic singularity (See \cite[Corollary 4.3]{H}).
By Proposition \ref{slc mld}, $\mathrm{mld}(0; A, (\mathrm{in}_{(2,1,1)}f))\ge 0$.
Therefore by Proposition \ref{mld initial}, 
$E_{(2,1,1)}$ computes $\mathrm{mld}(0; A, (f))$ and $\mathrm{mld}(0; A, (\mathrm{in}_{(2,1,1)}f))$ and 
we have
 $$\mathrm{mld}(0; A, (f))=\mathrm{mld}(0; A, (\mathrm{in}_{(2,1,1)}f))=0.$$

\vskip.3truecm
\noindent
{\bf Step 4.} Assume   $a_1=a_2=0$.\\
By a coordinates transformation 
$$x\mapsto x+b_1yz+b_2z^2\ \mathrm{for\ some}\ b_1,b_2\in k,$$  we may assume that 
$$\mathrm{in}_{(2,1,1)}f=x^2+b_3y^3z+b_4yz^3$$
for some  $b_3,b_4\in k$.
Then by a coordinate transformation 
$$y\mapsto c_1y+c_2z,\ z\mapsto c_3y+c_4z\ \mathrm{for\ some}\ c_1,c_2,c_3,c_4\in k,$$
  we may assume that 
$$\mathrm{in}_{(2,1,1)}f=x^2+c_5y^2z^2+c_6y^3z$$
for some  $c_5,c_6\in k$.
Then by a coordinate transformation 
$$x\mapsto x+dyz\ \mathrm{for\ some}\ d\in k,$$  we may assume that 
$$\mathrm{in}_{(2,1,1)}f=x^2+ey^3z$$
for some  $e\in k$.
Note that $\mathrm{ord}_{(15,8,6)}f=30$ since $\mathrm{in}_{(2,1,1)}f=x^2+ey^3z$.
Then for  the prime divisor $E_{(15,8,6)}$,
$a(E_{(15,8,6)};A,(f))=a(E_{(15,8,6)};A,(\mathrm{in}_{(15,8,6)}f))=-1$.
Hence 
$E_{(15,8,6)}$ computes $\mathrm{mld}(0; A, (f))$ and
$\mathrm{mld}(0; A, (\mathrm{in}_{(15,8,6)}f))$ and 
we have
 $$\mathrm{mld}(0; A, (f))=\mathrm{mld}(0; A, (\mathrm{in}_{(15,8,6)}f))=-\infty.$$

\end{proof}

The following lemma will be used in  the proof of Proposition \ref{key prop}.

\begin{lem}\label{ord 4 chk neq 2}
Let $A=\mathrm{Spec}k[[x,y,z]]$,  $0$ be the origin of $A$ and $f\in k[[x,y,z]]$.
Suppose that  $\mathrm{in}_{(1,1,1)}f=\mathrm{in}_{(3,2,2)}f=x^2$
 and $\mathrm{ch}k\neq 2$. 
Then there exist a $k$-automorphism  $\phi$ of  $k[[x,y,z]]$ and $w=(w_1,w_2,w_3)\in \mathbb N^3$ such that 
$E_w$ computes  $\mathrm{mld}\big(0;A,(\phi(f))\big)$ and $\mathrm{mld}\big(0;A,(\mathrm{in}_w\phi(f))\big)$,
$$\mathrm{mld}(0;A,(f))=\mathrm{mld}\big(0;A,(\mathrm{in}_w\phi(f))\big)$$
and $w$ is  one of the following:
\begin{align*}
\mathrm{(i)}&\ (10,5,4),& \mathrm{(ii)}&\ (2,1,1),& \mathrm{(iii)}&\ (15,8,6).
\end{align*}
In  case $\mathrm{(ii)}$, $\mathrm{mld}(0;A,(f))=0.$
In  case $\mathrm{(i)\ and\  (iii)}$, $\mathrm{mld}(0;A,(f))=-\infty.$
\end{lem}

\begin{proof}
Let $h$ be the element of $k[[x,y,z]]$ such that $f=x^2+h$.
Note that $\mathrm{ord}_{(2,1,1)}f\ge 4$ since
$\mathrm{in}_{(1,1,1)}f=\mathrm{in}_{(3,2,2)}f=x^2$.
Therefore $\mathrm{ord}_{(2,1,1)}h\ge 4$.

\noindent
{\bf Step 1.} Assume $\mathrm{ord}_{(2,1,1)}h\ge 5$.\\
In the same way as Step 1 in the proof of Lemma \ref{ord 4 chk = 2}, we can prove that
 $E_{(10,5,4)}$ computes $\mathrm{mld}(0; A, (f))$ and $\mathrm{mld}(0; A, (\mathrm{in}_{(10,5,4)}f))$ and we have
  $$\mathrm{mld}(0; A, (f))=\mathrm{mld}(0; A, (\mathrm{in}_{(10,5,4)}f))=-\infty.$$

\vskip.3truecm
Note that if   $\mathrm{ord}_{(2,1,1)}h=4$, we may assume that 
$$\mathrm{in}_{(2,1,1)}h=y^4,\ y^3z, \ y^2z^2,\   y^2z(y+z)\ \mathrm{or}\  yz(y+z)(y+az)\ \mathrm{for\ }\ a\neq 0,1.$$

In fact, since $\mathrm{ch}k\neq 2$, we may assume that $\mathrm{in}_{(2,1,1)}h$ is a homogeneous polynomial of $y$ and $z$.
Therefore, we can factor $\mathrm{in}_{(2,1,1)}h$ in linear terms.
Hence we may assume that for some $a_1,\dots,a_8\in k$,
$$\mathrm{in}_{(2,1,1)}h=(a_1y+a_2z)(a_3y+a_4z)(a_5y+a_6z)(a_7y+a_8z).$$
By a coordinates transformation
$$x\mapsto b_1x,\  y\mapsto b_2y+b_3z,\ z\mapsto b_4y+b_5z\ \mathrm{for\ some}\ b_1,\dots,b_5\in k,$$
  we may assume that  $\mathrm{for\ some}\ c\neq 0, a\neq 0,1$
$$\mathrm{in}_{(2,1,1)}cf=x^2+y^4,\ x^2+y^3z, \ x^2+y^2z^2,\   x^2+y^2z(y+z)\ \mathrm{or}\  x^2+yz(y+z)(y+az).$$
Since $\mathrm{mld}(0;A,(g))=\mathrm{mld}(0;A,(dg))$ for any $d\in k\setminus\{0\}$ and any $g\in(x,y,z)\setminus\{0\}$,
we may assume that 
$$\mathrm{in}_{(2,1,1)}h=y^4,\ y^3z, \ y^2z^2,\   y^2z(y+z)\ \mathrm{or}\  yz(y+z)(y+az)\ \mathrm{for\ }\ a\neq 0,1.$$

\vskip.3truecm
\noindent
{\bf Step 2.} Assume  $\mathrm{in}_{(2,1,1)}h=yz(y+z)(y+az)\ \mathrm{for\ }\ a\neq 0,1.$\\
Then $\mathrm{in}_{(2,1,1)}f$ is the defining equation of  a simple elliptic singularity (See  \cite[Corollary 4.3]{H} and \cite[Theorem 7.6.4]{I book}). 
By Proposition \ref{slc mld}, $\mathrm{mld}(0; A, (\mathrm{in}_{(2,1,1)}f))\ge 0$.
Therefore by Proposition \ref{mld initial}, 
$E_{(2,1,1)}$ computes $\mathrm{mld}(0; A, (f))$ and $\mathrm{mld}(0; A, (\mathrm{in}_{(2,1,1)}f))$ and we have
 $$\mathrm{mld}(0; A, (f))=\mathrm{mld}(0; A, (\mathrm{in}_{(2,1,1)}f))=0.$$

\vskip.3truecm
\noindent
{\bf Step 3.} Assume   $\mathrm{in}_{(2,1,1)}h=y^2z^2\ \mathrm{or}\ \  y^2z(y+z)$.\\
{\bf Step 3-1.}  Assume  $\mathrm{ch}k=0$.\\
$\mathrm{Spec}k[[x,y,z]]/(x^2+y^2z^2)$ is semi-log canonical at the origin by \cite[Main theorem]{LR}.
By Proposition \ref{slc mld}, $\mathrm{mld}(0; A, (x^2+y^2z^2))\ge0$.
Note that $\mathrm{in}_{(3,2,1)}(x^2+y^2z(y+z))=x^2+y^2z^2$.
By Proposition \ref{mld initial}, 
$$\mathrm{mld}(0; A, (f))\ge\mathrm{mld}(0; A, (\mathrm{in}_{(2,1,1)}f))\ge\mathrm{mld}(0; A, (x^2+y^2z^2))\ge0.$$
Therefore
$E_{(2,1,1)}$ computes $\mathrm{mld}(0; A, (f))$ and $\mathrm{mld}(0; A, (\mathrm{in}_{(2,1,1)}f))$ and
we have
  $$\mathrm{mld}(0; A, (f))=\mathrm{mld}(0; A, (\mathrm{in}_{(2,1,1)}f))=0.$$

\noindent
{\bf Step 3-2.} Assume   $\mathrm{ch}k\neq 0$.\\
Then the pair $(A, (x^2+y^2z^2))$ is F-pure at the origin by Proposition \ref{Fedder}.
Therefore $(A, (x^2+y^2z^2))$ is log canonical at the origin by Theorem \ref{HW thm}.
Hence in the same way as Step 3-1, we can prove  that 
$E_{(2,1,1)}$ computes $\mathrm{mld}(0; A, (f))$ and $\mathrm{mld}(0; A, (\mathrm{in}_{(2,1,1)}f))$ and
we have
  $$\mathrm{mld}(0; A, (f))=\mathrm{mld}(0; A, (\mathrm{in}_{(2,1,1)}f))=0.$$

\vskip.3truecm
\noindent
{\bf Step 4.} Assume   $\mathrm{in}_{(2,1,1)}h=y^3z$.\\
Note that $\mathrm{ord}_{(15,8,6)}f=30$ since $\mathrm{in}_{(2,1,1)}f=x^2+y^3z$.
Then for  the prime divisor $E_{(15,8,6)}$,
$a(E_{(15,8,6)};A,(f))=a(E_{(15,8,6)};A,(\mathrm{in}_{(15,8,6)}f))=-1$.
Hence 
$E_{(15,8,6)}$ computes $\mathrm{mld}(0; A, (f))$ and $\mathrm{mld}(0; A, (\mathrm{in}_{(15,8,6)}f))$ and
we have
  $$\mathrm{mld}(0; A, (f))=\mathrm{mld}(0; A, (\mathrm{in}_{(15,8,6)}f))=-\infty.$$

\vskip.3truecm
\noindent
{\bf Step 5.} Assume   $\mathrm{in}_{(2,1,1)}h=y^4$.\\
Note that $\mathrm{ord}_{(10,5,4)}f=20$ since $\mathrm{in}_{(2,1,1)}f=x^2+y^4$.
Then for  the prime divisor $E_{(10,5,4)}$,
$a(E_{(10,5,4)};A,(f))=a(E_{(10,5,4)};A,(\mathrm{in}_{(10,5,4)}f))=-1$.
Hence 
$E_{(10,5,4)}$ computes $\mathrm{mld}(0; A, (f))$ and $\mathrm{mld}(0; A, (\mathrm{in}_{(10,5,4)}f))$ and 
we have
 $$\mathrm{mld}(0; A, (f))=\mathrm{mld}(0; A, (\mathrm{in}_{(10,5,4)}f))=-\infty.$$

\end{proof}

\begin{prop}\label{key prop}
Let $A=\mathrm{Spec}k[[x,y,z]]$,  $0$ be the origin of $A$ and $f\in k[[x,y,z]]$.
Suppose that $\mathrm{ord}_{(1,1,1)}f=2$. 
Then there exist a $k$-automorphism  $\phi$ of $k[[x,y,z]]$ and $w=(w_1,w_2,w_3)\in \mathbb N^3$ such that 
$E_w$ computes  $\mathrm{mld}\big(0;A,(\phi(f))\big)$ and $\mathrm{mld}\big(0;A,(\mathrm{in}_w\phi(f))\big)$,
$$\mathrm{mld}(0;A,(f))=\mathrm{mld}\big(0;A,(\mathrm{in}_w\phi(f))\big)$$
and $w$ is  one of the following:
\begin{align*}
\mathrm{(i)}&\ (1,1,1),& \mathrm{(ii)}&\ (3,2,2),&\mathrm{(iii)}&\ (2,1,1),& \mathrm{(iv)}&\ (6,4,3),& \mathrm{(v)}&\ (9,6,4),\\
\mathrm{(vi)}&\ (15,10,6),& \mathrm{(vii)}&\ (3,2,1),&\mathrm{(viii)}&\ (10,5,4),& \mathrm{(ix)}&\ (15,8,6),& \mathrm{(x)}&\ (21,14,6).
\end{align*}
In  case $\mathrm{(i),(ii)},\dots,\mathrm{(vii)}$, $\mathrm{mld}(0;A,(f))\ge 0.$
In  case $\mathrm{(viii)}$, $\mathrm{(ix)}$,  $\mathrm{(x)}$, $\mathrm{mld}(0;A,(f))=-\infty.$
\end{prop}

\begin{proof}

\noindent
{\bf Step 1.} Let $w_1=(1,1,1)$. We will check  $\mathrm{in}_{w_1}(f)$.\\
\noindent
{\bf Step 1-1.} Assume  $\mathrm{ch}k\neq 2$.\\
Then there exists a $k$-automorphism  $\phi$ of $k[[x,y,z]]$ such that
$\mathrm{in}_{w_1}\phi(f)=x^2,\ x^2+y^2\ \mathrm{or}\ x^2+y^2+z^2$.
Note that
$$\mathrm{mld}(0; A, (x^2))=-\infty,\ \mathrm{mld}(0; A, (x^2+y^2))=\mathrm{mld}(0; A, (x^2+y^2+z^2))=1.$$
By Proposition \ref{mld initial}, if $\mathrm{in}_{w_1}(f)=x^2+y^2\ \mathrm{or}\ x^2+y^2+z^2$, then 
 $E_{w_1}$ computes $\mathrm{mld}\big(0;A,(f)\big)$ and $\mathrm{mld}\big(0;A,(\mathrm{in}_{w_1}f)\big)$ and
we have
 $$\mathrm{mld}(0;A,(f))=\mathrm{mld}\big(0;A,(\mathrm{in}_wf)\big)=1.$$

\noindent
{\bf Step 1-2.} Assume  $\mathrm{ch}k= 2$. \\
Then there exists a $k$-automorphism  $\phi$ of $k[[x,y,z]]$ such that
$\mathrm{in}_{w_1}\phi(f)=x^2,\ x^2+xy\ \mathrm{or}\ x^2+axy+bxz+yz\ \mathrm{for\ some}\ a,b\in k$.
Note that
$$\mathrm{in}_{(2,1,1)}(x^2+xy)=xy,\  \mathrm{in}_{(2,1,1)}( x^2+axy+bxz+yz)=yz,\  \mathrm{mld}(0; A, (xy))=1.$$
By Proposition \ref{mld initial}, we have 
$$\mathrm{mld}(0; A, (x^2+xy))=\mathrm{mld}(0; A, (x^2+axy+bxz+yz))=1.$$
By Proposition \ref{mld initial}, if $\mathrm{in}_{w_1}(f)=x^2+xy\ \mathrm{or}\ x^2+axy+bxz+yz$, then 
 $E_{w_1}$ computes $\mathrm{mld}\big(0;A,(f)\big)$ and $\mathrm{mld}\big(0;A,(\mathrm{in}_{w_1}f)\big)$ and
we have
 $$\mathrm{mld}(0;A,(f))=\mathrm{mld}\big(0;A,(\mathrm{in}_wf)\big)=1.$$

\vskip.3truecm
\noindent
{\bf Step 2.} Assume  $\mathrm{in}_{w_1}f=x^2$.\\
 Let $w_2=(3,2,2)$.
Note that $\mathrm{ord}_{w_2}f=6$ since $\mathrm{in}_{w_1}f=x^2$.
By a coordinates transformation
$$y\mapsto c_1y+c_2z,\ z\mapsto c_3y+c_4z\ \mathrm{for\ some}\ c_1,c_2,c_3,c_4\in k,$$
  we may assume that  
$$\mathrm{in}_{w_2}f=x^2,\ x^2+y^3,\ x^2+y^2z\ \mbox{or}\ x^2+yz(y+az)\ \ \mbox{for\ some}\ a\in k\setminus\{0\}.$$
Note that this coordinate transformation does not change $\mathrm{in}_{w_1}f=x^2$.

\noindent
{\bf Step 2-1.}
Assume  $\mathrm{in}_{w_2}f=x^2$.\\
In this case, this proposition holds by Lemma \ref{ord 4 chk = 2} and Lemma \ref{ord 4 chk neq 2}. 

\vskip.2truecm
\noindent
{\bf Step 2-2.} Assume   $\mathrm{in}_{w_2}f=x^2+y^2z$ or $x^2+yz(y+az)$ for some $a\in k\setminus\{0\}.$\\
Since $\mathrm{in}_{(2,1,2)}(x^2+yz(y+az))=x^2+y^2z$ and $\mathrm{mld}(0; A, (x^2+y^2z))=1$ (for example, see \cite[Theorem 4.8,(ii)]{IR2}),
we have
$\mathrm{mld}(0; A, (x^2+yz(y+az)))=1$ by Proposition \ref{mld initial}.
Therefore by Proposition \ref{mld initial}, $E_{w_2}$ computes $\mathrm{mld}(0;A,(f))$ and $\mathrm{mld}(0;A,(\mathrm{in}_{w_2}f))$ and we have
 $$\mathrm{mld}(0;A,(f))=\mathrm{mld}(0;A,(\mathrm{in}_{w_2}f))=1.$$

\vskip.3truecm
\noindent
{\bf Step 3.} Assume $\mathrm{in}_{w_1}f=x^2$ and $\mathrm{in}_{w_2}f=x^2+y^3$.\\
 Let $w_3=(6,4,3)$.
Note that $\mathrm{ord}_{w_3}f=12$ since $\mathrm{in}_{w_1}f=x^2$ and $\mathrm{in}_{w_2}f=x^2+y^3$.
Therefore 
$$\mathrm{in}_{w_3}f=x^2+y^3+a_1xz^2+a_2z^4\ \ \mathrm{for\ some}\ a_1,a_2\in k.$$
By a coordinates transformation
 $x\mapsto x+bz^2$ for some $b\in k$ with $b^2+a_1b+a_2=0,$
  we may assume that 
$\mathrm{in}_{w_3}f=x^2+y^3+cxz^2$ for some $c\in k$.
Therefore by a coordinates transformation $z\mapsto dz$ for some $d\in k\setminus \{0\}$,
 we may assume that 
$$\mathrm{in}_{w_3}f=x^2+y^3\ \ \mbox{or}\ \ x^2+y^3+xz^2.$$
Note that these coordinate transformation do not change $\mathrm{in}_{w_1}f=x^2$ and $\mathrm{in}_{w_2}f=x^2+y^3$.
Since $x^2+y^3+xz^2$ is the defining equation of a rational double point (See \cite[Section 3]{A}),
   by Theorem \ref{inversion of adjunction},  we have
$$\mathrm{mld}(0; A, (x^2+y^3+xz^2))=1.$$
Therefore by Proposition \ref{mld initial}, if $\mathrm{in}_{w_3}f=x^2+y^3+xz^2$, $E_{w_3}$ computes $\mathrm{mld}(0;A,(f))$ and $\mathrm{mld}(0;A,(\mathrm{in}_{w_3}f))$ and we have
 $$\mathrm{mld}(0;A,(f))=\mathrm{mld}(0;A,(\mathrm{in}_{w_3}f))=1.$$

\vskip.3truecm
\noindent
{\bf Step 4.} Assume $\mathrm{in}_{w_1}f=x^2$ and $\mathrm{in}_{w_i}f=x^2+y^3$ $(i=2,3)$.\\
 Let $w_4=(9,6,4)$.
Note that $\mathrm{ord}_{w_4}f=18$ since $\mathrm{in}_{w_1}f=x^2$ and $\mathrm{in}_{w_i}f=x^2+y^3$ $(i=2,3)$.
Therefore 
$$\mathrm{in}_{w_4}f=x^2+y^3+ayz^3\ \mathrm{for\ some}\ a\in k.$$
Therefore by a coordinates transformation
 $z\mapsto bz$ for some $b\in k\setminus\{0\}$,
we may assume that 
$$\mathrm{in}_{w_4}f=x^2+y^3\ \ \mbox{or}\ \ x^2+y^3+yz^3.$$
Note that this coordinate transformation does not change $\mathrm{in}_{w_1}f=x^2$ and $\mathrm{in}_{w_i}f=x^2+y^3$ $(i=2,3)$.
Since $x^2+y^3+yz^3$ is the defining equation of a rational double point (See \cite[Section 3]{A}),
  by Theorem \ref{inversion of adjunction},  we have
$$\mathrm{mld}(0; A, (x^2+y^3+yz^3))=1.$$
Therefore by Proposition \ref{mld initial}, if $\mathrm{in}_{w_4}f=x^2+y^3+yz^3$, $E_{w_4}$ computes $\mathrm{mld}(0;A,(f))$ and $\mathrm{mld}(0;A,(\mathrm{in}_{w_4}f))$ and we have
 $$\mathrm{mld}(0;A,(f))=\mathrm{mld}(0;A,(\mathrm{in}_{w_4}f))=1.$$

\vskip.3truecm
\noindent
{\bf Step 5.} Assume $\mathrm{in}_{w_1}f=x^2$ and $\mathrm{in}_{w_i}f=x^2+y^3$ $(i=2,3,4)$.\\
 Let $w_5=(15,10,6)$.
Note that $\mathrm{ord}_{w_5}f=30$ since $\mathrm{in}_{w_1}f=x^2$ and $\mathrm{in}_{w_i}f=x^2+y^3$ $(i=2,3,4)$.
Therefore 
$$\mathrm{in}_{w_5}f=x^2+y^3+az^5\ \mathrm{for\ some}\ a\in k.$$
Therefore by a coordinates transformation
 $z\mapsto bz$ for some $b\in k\setminus\{0\}$,
we may assume that 
$$\mathrm{in}_{w_5}f=x^2+y^3\ \mbox{or}\ x^2+y^3+z^5.$$
Note that this coordinate transformation does not change $\mathrm{in}_{w_1}f=x^2$ and $\mathrm{in}_{w_i}f=x^2+y^3$ $(i=2,3,4)$.
Since $x^2+y^3+z^5$ is the defining equation of a rational double point (See \cite[Section 3]{A}),
 by Theorem \ref{inversion of adjunction},  we have
$$\mathrm{mld}(0; A, (x^2+y^3+z^5))=1.$$
Therefore by Proposition \ref{mld initial}, if $\mathrm{in}_{w_5}f=x^2+y^3+z^5$, $E_{w_5}$ computes $\mathrm{mld}(0;A,(f))$ and $\mathrm{mld}(0;A,(\mathrm{in}_{w_5}f))$ and we have
 $$\mathrm{mld}(0;A,(f))=\mathrm{mld}(0;A,(\mathrm{in}_{w_5}f))=1.$$

\vskip.3truecm
\noindent
{\bf Step 6.} Assume  $\mathrm{in}_{w_1}f=x^2$ and $\mathrm{in}_{w_i}f=x^2+y^3$ $(i=2,\dots,5)$.\\
 Let $w_6=(3,2,1)$.
Note that $\mathrm{ord}_{w_6}f=6$ since $\mathrm{in}_{w_1}f=x^2$ and $\mathrm{in}_{w_i}f=x^2+y^3$ $(i=2,\dots,5)$.
Therefore  for some $a_1,\dots,a_5\in k$,
$$\mathrm{in}_{w_6}f=x^2+y^3+a_1xyz+a_2xz^3+a_3z^6+a_4yz^4+a_5y^2z^2.$$

\noindent
{\bf Step 6-1.} Assume  $\mathrm{ch}k=2$.

\noindent
{\bf Step 6-1-1.} Assume $a_1\neq 0$. \\
Then the pair $(A, (\mathrm{in}_{w_6}f))$ is $F$-pure at the origin by Proposition \ref{Fedder}.
Therefore $(A, (\mathrm{in}_{w_6}f))$ is log canonical at the origin by Theorem \ref{HW thm}.
Hence by Proposition \ref{mld initial}, $E_{w_6}$ computes $\mathrm{mld}(0;A,(f))$ and $\mathrm{mld}(0;A,(\mathrm{in}_{w_6}f))$ and we have
 $$\mathrm{mld}(0;A,(f))=\mathrm{mld}(0;A,(\mathrm{in}_{w_6}f))=0.$$

\noindent
{\bf Step 6-1-2.} Assume $a_1=0$ and $a_2\neq 0$. \\
By a coordinates transformation $y\mapsto y+bz^2$ for some $b\in k$,  we may assume that 
$a_4=0$.
By a coordinates transformation $x\mapsto x+cz^3$ for some $c\in k$,  we may assume that 
$a_3=0$.
Therefore we may assume that for some $d\in k$,
$$\mathrm{in}_{w_6}f=x^2+y^3+a_2xz^3+dy^2z^2.$$
Since
 $\mathrm{in}_{w_6}f$ is the defining equation of  a simple elliptic singularity (See  \cite[Corollary 4.3]{H}),
 we have $\mathrm{mld}(0;A,(\mathrm{in}_{w_6}f))\ge 0$ by Proposition \ref{slc mld}.
Thus by Proposition \ref{mld initial},  $E_{w_6}$ computes $\mathrm{mld}(0;A,(f))$ and $\mathrm{mld}(0;A,(\mathrm{in}_{w_6}f))$ and we have
 $$\mathrm{mld}(0;A,(f))=\mathrm{mld}(0;A,(\mathrm{in}_{w_6}f))=0.$$

\noindent
{\bf Step 6-1-3.} Assume $a_1=a_2=0$. \\
By a coordinates transformation $y\mapsto y+bz^2$ for some $b\in k$,  we may assume that 
$a_4=0$.
By a coordinates transformation $x\mapsto x+c_1z^3+c_2yz$ for some $c_1,c_2\in k$,  we may assume that 
$a_3=a_5=0$.
Therefore 
  we may assume that $\mathrm{in}_{w_6}f=x^2+y^3$.
Note that these coordinate transformations do not change $\mathrm{in}_{w_1}f=x^2$ and $\mathrm{in}_{w_i}f=x^2+y^3$ $(i=2,\dots,5)$.

\vskip.2truecm
\noindent
{\bf Step 6-2.} Assume $\mathrm{ch}k\neq2$. \\
By a coordinates transformation $x\mapsto x+b_1yz+b_2z^3$ for some $b_1,b_2\in k$, 
 we may assume that $a_1=a_2=0$.
Moreover, by a coordinates transformation $y\mapsto y+cz^2$ for some $c\in k$ with $c^3+a_5c^2+a_4c+a_3=0$,  we may assume that 
$$\mathrm{in}_{w_6}f=x^2+y^3+dyz^4+ey^2z^2$$ for some $d,e\in k$.
Therefore 
$$\mathrm{in}_{w_6}f=x^2+y(y-\alpha z^2)(y-\beta z^2)$$ for some $\alpha,\beta\in k$.
Note that these coordinate transformation do not change $\mathrm{in}_{w_1}f=x^2$ and $\mathrm{in}_{w_i}f=x^2+y^3$ $(i=2,\dots,5)$.

If $\mathrm{in}_{w_6}f\neq x^2+y^3$,
we may assume that  $\alpha\neq 0$. 
Then by a coordinates transformation
 $z\mapsto \gamma z$ for some $\gamma\in k\setminus\{0\}$,
we may assume that 
$$\mathrm{in}_{w_6}f=x^2+y(y-z^2)(y-\delta z^2)$$ for some $\delta\in k$.

\noindent
{\bf Step 6-2-1.}
Assume $\delta\neq 0,1$.\\  
Then $\mathrm{in}_{w_6}f$ is the defining equation of  a simple elliptic singularity (See  \cite[Corollary 4.3]{H} and \cite[Theorem 7.6.4]{I book}). 
By Proposition \ref{slc mld}, $\mathrm{mld}(0;A,(\mathrm{in}_{w_6}f))\ge 0$.
Thus by Proposition \ref{mld initial},  $E_{w_6}$ computes $\mathrm{mld}(0;A,(f))$ and $\mathrm{mld}(0;A,(\mathrm{in}_{w_6}f))$ and we have
 $$\mathrm{mld}(0;A,(f))=\mathrm{mld}(0;A,(\mathrm{in}_{w_6}f))=0.$$

\noindent
{\bf Step 6-2-2.}
Assume $\mathrm{ch}k= 0$ and $\delta=0$ or $\delta=1$.\\
 Then 
 $\mathrm{Spec}k[[x,y,z]]/\mathrm{in}_{w_6}f$ is semi-log canonical at the origin by \cite[Main Theorem]{LR}.
By Proposition \ref{slc mld}, $\mathrm{mld}(0;A,(\mathrm{in}_{w_6}f))\ge 0$.
Therefore by Proposition \ref{mld initial}, $E_{w_6}$ computes $\mathrm{mld}(0;A,(f))$ and $\mathrm{mld}(0;A,(\mathrm{in}_{w_6}f))$ and we have
 $$\mathrm{mld}(0;A,(f))=\mathrm{mld}(0;A,(\mathrm{in}_{w_6}f))=0.$$

\noindent
{\bf Step 6-2-3.}
Assume $\mathrm{ch}k\neq 0$ and $\delta=0$ or $\delta=1$.\\
The pair $(A, (\mathrm{in}_{w_6}f))$ is $F$-pure at the origin by  Proposition \ref{Fedder}.
The pair $(A, (\mathrm{in}_{w_6}f))$ is log canonical at the origin by Theorem \ref{HW thm}.
Therefore by Proposition \ref{mld initial},  $E_{w_6}$ computes $\mathrm{mld}(0;A,(f))$ and $\mathrm{mld}(0;A,(\mathrm{in}_{w_6}f))$ and we have
 $$\mathrm{mld}(0;A,(f))=\mathrm{mld}(0;A,(\mathrm{in}_{w_6}f))=0.$$

\vskip.3truecm
\noindent
{\bf Step 7.} Assume $\mathrm{in}_{w_1}f=x^2$ and $\mathrm{in}_{w_i}f=x^2+y^3$ $(i=2,\dots,6)$.\\
 Let $w_7=(21,14,6)$.
Note that $\mathrm{ord}_{w_7}f=42$ since $\mathrm{in}_{w_1}f=x^2$ and $\mathrm{in}_{w_i}f=x^2+y^3$ $(i=2,\dots,6)$.
Then for  the prime divisor $E_{w_7}$,
$a(E_{w_7};A,(f))=a(E_{w_7};A,(\mathrm{in}_{w_7}f))=-1$.
Therefore $E_{w_7}$ computes $\mathrm{mld}(0; A, (f))$ and $\mathrm{mld}(0; A, (\mathrm{in}_{w_7}f))$ and we have
  $$\mathrm{mld}(0; A, (f))=\mathrm{mld}(0; A, (\mathrm{in}_{w_7}f))=-\infty.$$

\end{proof}

\begin{thm}\label{key thm}
Let $A=\mathrm{Spec}k[[x,y,z]]$,  $0$ be the origin of $A$ and $f\in k[[x,y,z]]\setminus\{0\}$.
Then there exist a $k$-automorphism  $\phi$ of $k[[x,y,z]]$ and $w=(w_1,w_2,w_3)\in \mathbb N^3$ such that 
$$\mathrm{mld}(0;A,(f))=\mathrm{mld}\big(0;A,(\mathrm{in}_w\phi(f))\big)$$
and $w$ is  one of the following:
\begin{align*}
\mathrm{(i)}&\ (1,1,1),& \mathrm{(ii)}&\ (3,2,2),&\mathrm{(iii)}&\ (2,1,1),& \mathrm{(iv)}&\ (6,4,3),& \mathrm{(v)}&\ (9,6,4),\\
\mathrm{(vi)}&\ (15,10,6),& \mathrm{(vii)}&\ (3,2,1),&\mathrm{(viii)}&\ (10,5,4),& \mathrm{(ix)}&\ (15,8,6),& \mathrm{(x)}&\ (21,14,6).
\end{align*}
Moreover if $\mathrm{ord}_{(1,1,1)}f\neq 3$ or $\mathrm{mld}(0;A,(\mathrm{in}_{(1,1,1)}f))\ge 0$, then
$E_w$ computes  $\mathrm{mld}\big(0;A,(\phi(f))\big)$ and $\mathrm{mld}\big(0;A,(\mathrm{in}_w\phi(f))\big)$.
In  case $\mathrm{(ii),(iii)},\dots,\mathrm{(vii)}$, $\mathrm{mld}(0;A,(f))\ge 0.$
In  case $\mathrm{(viii)}$, $\mathrm{(ix)}$,  $\mathrm{(x)}$, $\mathrm{mld}(0;A,(f))=-\infty.$
\end{thm}

\begin{proof}
If $\mathrm{ord}_{(1,1,1)}f=0$, then $$\mathrm{mld}(0; A, \mathcal O_A)=\mathrm{mld}(0; A, (f))=\mathrm{mld}(0; A, (\mathrm{in}_{(1,1,1)}f))=3.$$
If $\mathrm{ord}_{(1,1,1)}f=1$, then $$\mathrm{mld}(0; A, (f))=\mathrm{mld}(0; A, (\mathrm{in}_{(1,1,1)}f))=2.$$
If  $\mathrm{ord}_{(1,1,1)}f\ge 2$, then this theorem holds by Lemma \ref{slc order 3} and Proposition \ref{key prop}.
\end{proof}

\begin{thm}
Let  $X=\mathrm{Spec}k[[x,y,z]]/(f)$ and $f\in (x,y,z)$. 
Then the following are equivalent:
\begin{enumerate}
\item $X$ is semi-log canonical at the origin,
\item There exist a $k$-automorphism  $\phi$ of  $k[[x,y,z]]$ and $w=(w_1,w_2,w_3)\in \mathbb N^3$ such that 
$\mathrm{Spec}k[[x,y,z]]/(\mathrm{in}_w\phi(f))$ is semi-log canonical at the origin and $w$ is  one of the following:
\begin{align*}
\mathrm{(i)}&\ (1,1,1),& \mathrm{(ii)}&\ (3,2,2),& \mathrm{(iii)}&\ (2,1,1),& \mathrm{(iv)}&\ (6,4,3),\\
 \mathrm{(v)}&\ (9,6,4),&\mathrm{(vi)}&\ (15,10,6),& \mathrm{(vii)}&\ (3,2,1).
\end{align*}
\end{enumerate}
\end{thm}

\begin{proof}
By Proposition \ref{slc mld} and Proposition \ref{mld initial}, $(2)$ implies $(1)$.

Next we prove the converse $(1)\Rightarrow (2)$.
Assume that $X$ is semi-log canonical at the origin.
Let $A=\mathrm{Spec} k[[x,y,z]]$ and $0$ be the origin of $A$.
By Proposition \ref{slc mld}, it is enough to prove that there exist a $k$-automorphism  $\phi$ of $k[[x,y,z]]$ and $w=(w_1,w_2,w_3)\in \mathbb N^3$ such that   $\mathrm{mld}(0; A, (f))=\mathrm{mld}(0; A, (\mathrm{in}_w\phi(f)))$ and $w$ is one of $\mathrm{(i)},\dots,\mathrm{(vii)}$.
Therefore this theorem holds by Theorem \ref{key thm}.

\end{proof}

The following example shows that the statement in Theorem \ref{key thm} does not hold for an ideal with the exponent $\neq 1$.

\begin{ex}
Let $A=\mathrm{Spec} \mathbb C[[x,y,z]]$, $0$ be the origin of $A$, $t=7/10$ and $f=(x^2+y^2+z^2)^2+az^5+by^5+cz^5\in\mathbb C[[x,y,z]]$, where $a,b$ and $c$ are general complex numbers.
Then  we have
$$\mathrm{mld}(0;A,(f)^t)=0\ \ \  \mbox{and}\ \ \ \ a(E_w;A,(\phi(f))^t)>0$$
for any $\mathbb C$-automorphism  $\phi$ of  $\mathbb C[[x,y,z]]$ and any $w\in \mathbb N^3$
(See \cite[Example 6.45]{KSC}).
\end{ex}

\section{Application}
In this section, we prove Conjecture $D_2$ in arbitrary characteristic.

\begin{defn} 
\label{zar}
Let $A$ be a smooth variety and $E$ be a prime divisor over $A$.
Then there is a sequence of blow-ups
$$A^{(n)}\stackrel{\varphi_n}\longrightarrow A^{(n-1)}\to\cdots\to
A^{(1)}\stackrel{\varphi_1}\longrightarrow A^{(0)}=A$$
such that 
\begin{enumerate}
    \item $E$ appears on $A^{(n)}$, i.e., the 
    center of $E$ on $A^{(n)}$ is of codimension 1 and 
    $A^{(n)}$ is normal at the generic point $p_n$ of  $E$,  
  
    \item $\varphi_i(p_i)=p_{i-1}$ for $1\leq i\leq n$, and
    \item $\varphi_i$ is the blow-up with the center $\overline{\{p_{i-1}\}}$.
\end{enumerate}
The  minimal number
 $i$ such that $\mathrm{codim} \overline{\{p_{i}\}}=1$ is denoted by $b(E)$.
\end{defn}

\begin{rem}
Zariski proved that every divisor can be reached by a sequence of blow-ups (for example, see \cite[Lemma 2.45]{KM}).
\end{rem}

\begin{defn}
Let $A=\mathrm{Spec}k[x_1,\dots,x_{n+c}]$, $I=(f_1,\dots,f_r)\subset(x_1\dots,x_{n+c})$ be an ideal of $k[x_1,\dots,x_{n+c}]$, $X=\mathrm{Spec}k[x_1,\dots,x_{n+c}]/I$ be an $n$-dimensional variety and $0$ be the origin of $X$.
We define $s_m(0;X)$ on $m\in \mathbb Z_{\ge 0}$ as follows:
$$s_m(0;X)$$
$$:=(m+1)n-\mathrm{dim}\ \mathrm{Spec}k\Big[\ x_i^{(j)}\ \Big|
\begin{array}{l}
1\le i\le n+c, \\
0\le j \le m
\end{array}
\Big]\Big/
\Big(\ x_1^{(0)},\dots,x_{n+c}^{(0)}, f_s^{(l)}\ \Big|
\begin{array}{l}
1\le s\le r, \\
0\le l\le m
\end{array}
\Big) .$$

We can define in the same way $s_m(0;Y)$ for $Y=\mathrm{Spec}k[[x_1,\dots,x_{n+c}]]/I$. 
\end{defn}

In \cite{I}, Ishii posed the following conjectures.

\begin{conj}[Conjecture $U_n$] 
   For every $n\in \mathbb N$ there is an integer $B_n\in \mathbb N$ depending only on $n$
   such that for every singularity $(X,x)$ of dimension $n$ 
   embedded into a smooth variety $A$ with $\mathrm{dim} A=\mathrm{emb}(X,x)$,
   there is a prime divisor $E$ over $A$ computing $\mathrm{mld} (x; A, I_X^c)$ and satisfying 
   $b(E)\leq B_n$.
   Here, $\mathrm{emb}(X,x)$ is the embedding dimension of $X$ at $x$, $I_X$ is the defining ideal of $X$  and $c=\mathrm{emb}(X,x)-n$.
   
\end{conj}

\begin{conj}[Conjecture $C_n$] 
\label{Cn}
For $n\in \mathbb N$,
there exists $N_n\in \mathbb N$ depending only on $n$ such that 
for  any $n$-dimensional variety $X=\mathrm{Spec}k[x_1,\dots,x_{n+c}]/I$ with $I\subset(x_1,\dots,x_{n+c})$, there exists
$m\leq N_n$ satisfying either 
$$\left\{\begin{array}{l}
s_m(0;X)=\mathrm{mld}(0;\mathrm{Spec}k[x_1,\dots,x_{n+c}],I^c )\geq 0, \mbox{or}\\
\\
s_m(0;X)<0, \ \ \mbox{when}\  \mathrm{mld}(0;\mathrm{Spec}k[x_1,\dots,x_{n+c}],I^c )=-\infty, \\
\end{array} \right. $$
where $0$ is the origin of $\mathrm{Spec}k[x_1,\dots,x_{n+c}].$

\end{conj}

\begin{conj}[Conjecture $D_n$] 
\label{Dn}
 For  $n\in \mathbb N$, there exists $M_n\in \mathbb N$ depending only on $n$
such that 
for any $n$-dimensional variety $X$ and a closed point $x\in X$ with a closed immersion $X\subset A$ around $x$ into a smooth variety $A$ of dimension $\mathrm{emb}(X,x)\le 2n$,
 there exists a prime divisor $E$ over $A$ with the
center at $x$ such that  $k_E \leq M_n$ and
$E$ computes $\mathrm{mld} (x; A, I_X^c)$.
 Here, $\mathrm{emb}(X,x)$ is the embedding dimension of $X$ at $x$, $I_X$ is the defining ideal of $X$ and $c=\mathrm{emb}(X,x)-n$.

\end{conj}

\begin{rem}
In \cite[Conjecture $D_n$]{I}, it is not assumed that $\mathrm{dim}A=\mathrm{emb}(X,x)$. 
However, this assumption is used in  the proof in \cite[Proof of Theorem 6.3]{I}.
 So in this paper we modify Conjecture $D_n$ according to the proof in \cite{I}.

\end{rem}

 The three conditions above are equivalent.
\begin{prop}\cite[Proposition 1.7]{I}  \label{Cn=Dn} Conjecture $U_n$, Conjecture $C_n$ and Conjecture $D_n$ are equivalent. \end{prop}

The following statements will
be used for the proof of Theorem \ref{D_2 arbitrary}.
\begin{prop}\cite[See the proof of Proposition 3.9]{I} \label{prop 3.9 in I}
Let $A$ be a smooth variety, $x$ be a closed point of $A$ and $E$ be a prime divisor over $A$ with $c_A(E)=x$. 
Then $b(E)\le k_E-\mathrm{dim}A+1$.
\end{prop}

\begin{lem}\cite[Lemma 3.5]{I}\label{nu=mu}  
Let $X$ be a variety and $x\in X$ a closed point.
We assume that $X$ is embedded into a smooth variety $A$ with codimension $c$.
Let
 $$r=      \min \left\{ \mathrm{ord}_E(I_X) \left| 
\begin{array}{l}  
\mbox{prime\ divisors}\ E\  \mbox{over\ $A$\ computing}\  \mathrm{mld} (x; A, I_X^c)\\
 \end{array} \right\}.\right.$$  
Then 
$$\begin{array}{l}
s_{r-1}(0;X)=\mathrm{mld}(0;\mathrm{Spec}k[x_1,\dots,x_{n+c}],I^c )\geq 0, \mbox{or}\\
s_{r-1}(0;X)<0, \ \ \mbox{when}\  \mathrm{mld}(0;\mathrm{Spec}k[x_1,\dots,x_{n+c}],I_X^c )=-\infty, \\
\end{array} $$
\end{lem}

\begin{thm}\label{D_2 arbitrary}
 Conjecture $U_2$, Conjecture $C_2$ and Conjecture $D_2$ hold and we can take $B_2\leq 39$, $N_2\leq 41$ and $M_2\leq 58$ in arbitrary characteristic.
\end{thm}

\begin{proof}
Let $X$ be a $2$-dimensional variety and $x$ be a closed point of $X$.

First we consider the case where $X$ is a hypersurface double point. 
Let $A=\mathrm{Spec}k[[x,y,z]]$, $0$ be the origin of $A$ and $f$ be an element of $(x,y,z)$ with $\mathrm{ord}_{(1,1,1)}f=2$.
We assume that $\widehat{\mathcal O}_{X,x}=\mathrm{Spec}k[[x,y,z]]/(f)$.
By Proposition \ref{key prop}, there exists a prime divisor $E$ over $A$  such that 
$k_E\le 40$ and $E$ computes $\mathrm{mld}(0;A,(f))$.
Therefore  Conjecture $D_2$ holds and we can take $M_2\le 40$.
Conjecture $U_2$ holds and  we can take $B_2\leq 39$ by Proposition \ref{prop 3.9 in I}.
Note that if $R\cong k[[x_1,\dots,x_{n+c}]]/I$, then 
$s_m(0;\mathrm{Spec}k[x_1,\dots,x_{n+c}]/I)=s_m(0;\mathrm{Spec}R)$.
Hence by  Lemma \ref{nu=mu},   Conjecture $C_2$ holds and we can take $N_2\le 41$.

If $X$ is not a hypersurface double point,  Conjecture $U_2$, Conjecture $C_2$  and Conjecture $D_2$ hold and  we can take  $B_2\leq 39$, $N_2\leq 41$ and $M_2\le 58$ by Case 1, Case 2,\dots, Case 5 in \cite[Proof of Theorem 6.3]{I}  (See \cite[Remark 6.6]{I}).

Therefore this theorem holds.
\end{proof}


%

\end{document}